\documentclass[12pt]{amsart}
\usepackage{amssymb,latexsym, array, pictex}
\usepackage{verbatim,euscript,adjustbox}
\usepackage{fullpage,color}

\usepackage{tikz}
\usetikzlibrary{arrows,shapes,trees} 

\usepackage[colorlinks=true,linkcolor=blue,urlcolor=my_color,citecolor=magenta]{hyperref}

\definecolor{my_color}{rgb}{0,0.5,0.5}
\definecolor{MIXT}{rgb}{0.4,0.3,0.6}

\input {cyracc.def}
\numberwithin{equation}{section}

\newtheorem{thm}{Theorem}[section]
\newtheorem{lm}[thm]{Lemma} 

\newtheorem{prop}[thm]{Proposition}

\theoremstyle{remark}

\newtheorem{rmk}[thm]{Remark}

\theoremstyle{definition}
\newtheorem{ex}[thm]{Example}

\newcommand {\g}{{\mathfrak g}}

\newcommand {\el}{{\mathfrak l}}

\newcommand {\p}{{\mathfrak p}}
\newcommand {\q}{{\mathfrak q}}

\newcommand {\te}{{\mathfrak t}}

\newcommand {\gln}{{\mathfrak{gl}}_n}
\newcommand {\sln}{{\mathfrak{sl}}_n}

\newcommand {\esi}{\varepsilon}
\newcommand {\ap}{\alpha}

\newcommand{\Gc}{\Gamma^{\sf C}}

\newcommand {\ad}{{\mathrm{ad}}}

\newcommand {\ind}{{\mathsf{ind\,}}}

\newcommand {\rk}{{\mathsf{rk\,}}}

\newcommand{\sig}{{\langle\sigma\rangle}}
\newcommand {\e}{\boldsymbol{e}}

\newcommand {\beq}{\begin{equation}}
\newcommand {\eeq}{\end{equation}}

\newcommand {\bbk}{\Bbbk}

\newcommand{\gt}{\mathfrak}
\newcommand{\wb}{\widehat{\gt b}}

\newcommand{\id}{{\rm id}}

\newcommand {\cE}{{\mathcal E}}

\newcommand {\cv}{{\mathcal V}}
\newcommand {\cK}{{\mathcal K}}
\newcommand {\cL}{{\mathcal L}}

\newcommand {\BZ}{{\mathbb Z}}

\newcommand {\Z}{{\mathbb Z}}
\newcommand {\BR}{{\mathbb R}}
\newcommand{\wg}{{\widehat{\gt g}}}
\newcommand{\wq}{{\widehat{\gt q}}}

\renewcommand{\le}{\leqslant}
\renewcommand{\ge}{\geqslant}

\newfam\eusfam%
\font\euszw=eusm10 scaled 1200%
\font\eusac=eusm7 scaled 1200%
\font\eusacc=eusm7 scaled 1000%
\textfont\eusfam=\euszw\scriptfont\eusfam=\eusac%
\scriptscriptfont\eusfam=\eusacc%
\begin{document}
\hfill {\scriptsize April 2, 2025} 
\vskip1ex

\title[Combinatorial interpretation of the index]{A combinatorial approach to the index \\
 of  seaweed subalgebras of 
Kac--Moody algebras}
\author[O.\,Yakimova]{Oksana Yakimova}
\address[O.Y.]{Institut f\"ur Mathematik, Friedrich-Schiller-Universit\"at Jena,  07737 Jena,
Deutschland}
\email{oksana.yakimova@uni-jena.de}
\begin{abstract}
In 2000, Dergachev and Kirillov introduced subalgebras of "seaweed type" in $\gt{gl}_n$ (or $\gt{sl}_n$) and computed their index using certain graphs. 
Then seaweed subalgebras $\gt q\subset\gt g$ were defined by Panyushev for any reductive $\gt g$. 
A few years later Joseph generalised this notion to the setting of (untwisted) affine Kac--Moody algebras $\wg$. Furthermore, he proved that the index of such a seaweed can be computed by the same formula that had been known for $\gt g$. In this paper, we construct graphs that 
help to understand the index of a seaweed $\gt q\subset\wg$, where $\wg$ is of affine  type {\sf A} or {\sf C}. 
\end{abstract}
\maketitle

\section{Introduction}   
\label{sect:intro}

\noindent 
Since the pioneering work of Dergachev and 
Kirillov~\cite{dk00}, {\it seaweed subalgberas} 
of reductive Lie algebras have attracted a great deal of attention
\cite{Dima01,ty-AIF,Dima03-b,jos,jos2,MY12,SW-C,mD}.  The seaweeds form a 
wide class of Lie algebras, which includes all parabolics and their Levi subalgebras. 
Unlike an arbitrary Lie algebra, it is reasonable to study the coadjoint action of a seaweed. An important 
numerical characterisation of the  
coadjoint representation is the {\it index}. In this paper, we study  the index of a  finite-dimensional seaweed in the Kac--Moody setting 
from a combinatorial point of view. 

The index of a Lie algebra $\q$, $\ind\q$, is the minimal dimension of a stabiliser for the 
coadjoint representation of $\q$. It  can be regarded as a generalisation of the notion of rank. That is, 
$\ind\q$ equals the rank of $\q$, if $\q$ is reductive.  If $\gt q$ is non-reductive, then a calculation of $\ind\gt q$ can be difficult. 
Nevertheless, there are classes of Lie algebras, for which the index can be determined. 
The seaweeds comprise one of these classes. 

For $\gln$, the seaweed subalgebras have been introduced in \cite{dk00} and 
their index 
has been computed using certain {graphs},
which are said to be type-{\sf A} meander graphs. 
A general definition suited for arbitrary reductive Lie algebras 
$\g$ appears in~\cite{Dima01}. Namely, if $\p_1,\p_2\subset \g$ are parabolic subalgebras such that 
$\p_1+\p_2=\g$, then $\q=\p_1\cap\p_2$ is called a seaweed in $\g$. (For this reason, some people 
began to use later the term ``biparabolic subalgebra" for such $\q$.) 

An inductive (reduction) procedure for computing the index of the seaweeds in the classical Lie 
algebras is introduced in~\cite{Dima01}. It helps to answer some subtle questions on the coadjoint 
action~\cite{Dima03-b,MY12}.  A general algebraic formula for the index of the seaweeds had been proposed 
by Tauvel and Yu
\cite[Conj.\,4.7]{ty-AIF} and then proved by Joseph \cite[Section\,8]{jos}. For these reasons we call it the 
Tauvel--Yu--Joseph ( {\it TYJ} for short) formula. 
Then it was shown in \cite{jos2} that the same formula applies in the setting of (untwisted) affine Kac--Moody algebras $\wg$. 

Let now $\gt p_1\subset\wg$ be a proper  standard parabolic subalgebra and $\gt p_2\subset\wg$ the opposite of 
a proper  standard parabolic. Then $\wq=\gt p_1\cap\gt p_2$ is a finite-dimensional Lie algebra, see  Section~\ref{KM} for an explanation. We say that $\wq$ is a seaweed in $\wg$. 
If $\wg$ is of affine type {\sf A}, then $\wq$ has an easy description as a semi-direct product, see Proposition~\ref{pA}.  

By  construction, $\wg$ is a $2$-dimensional extension of the loop algebra $\gt g[t,t^{-1}]=\gt g\otimes\bbk[t,t^{-1}]$,
see~\eqref{KM-sum}. The central part of this extension,  $\bbk C$,  is contained in $\wq$, but not in 
$[\wq,\wq]$  \cite[7.17\,\&\,7.20]{jos2}. Therefore there is a subalgebra $\bar\q\subset\wq$ such that 
$\bar\q\cong\wq/\bbk C$ and $\wq=\bar\q\oplus\bbk C$. Clearly $\ind\wq=1+\ind\bar\q$. 
For $\gt g$ equal to $\gt{sl}_n$ or $\gt{sp}_{2n}$, we interpret $\ind\bar\q$ in terms of the 
type-{\sf A} and type-{\sf C} {\it affine meander graphs} $\Gamma$ and $\Gc$. Our proofs are based on the TYJ formula. 
 
The graphs $\Gamma$ and $\Gc$ have $n$ and $2n$ vertices, respectively, which are placed on a circle.  
Some edges are drawn outside of the circle and the other ones inside. The outside edges are defined by the first parabolic  
$\gt p_1$ and the inside ones by  $\gt p_2$. A particular attention must be paid to the position of the  centre $o$ of the circle, see 
Figure~\ref{rules}. Then one counts the numbers of segments and cycles in the graph and checks, whether $o$ lies inside of a cycle. 
This information is used in a combinatorial formula for the index, see Theorems~\ref{thm-A} and \ref{thmC}.  
 
Our combinatorial approach simplifies calculations of $\ind\bar\q$.  For example, if $\wq$ is the intersection of two maximal standard parabolics in $\widehat{\gt{sl}_n}$, then 
$$
\bar{\gt q}=(\gt s(\gt{gl}_d\oplus\gt{gl}_{n-d})\oplus\bbk d)\ltimes (2\bbk^d{\otimes}(\bbk^{n-d})^*),
$$ 
where 
$2\bbk^d{\otimes}(\bbk^{n-d})^*=\bbk^d{\otimes}(\bbk^{n-d})^*\oplus \bbk^d{\otimes}(\bbk^{n-d})^*$ is an Abelian ideal of $\bar\q$ and  $\ad(d)$ acts on its direct summands with eigenvalues $0$ and $-1$,
see Example~\ref{ex-max}.
We have 
$$
\ind\bar\q=\gcd(n,2d)-\iota,
$$
 where 
$\iota=0$ if $\gcd(n,2d)$ divides $d$ and $\iota=2$ otherwise, see Example~\ref{1}. The formula for $\iota$ shows that describing $\iota$ is a non-trivial task, which would be more difficult without a graph. 
 
 In Section~\ref{fin},  we explain that  the TYJ formula can be used for justifications of the graph interpretations for $\ind\gt q$, where $\gt q\subset\gt g$ is a usual, not affine, seaweed. The original type {\sf A} approach of \cite{dk00} is generalised
 in \cite{SW-C,mD} to all classical types. Although the TYJ formula is used in \cite{mD},  proofs in \cite{dk00,SW-C,mD}  rely heavily on the  inductive procedures, in particular, 
 on the one developed in \cite{Dima01}.

 Throughout the paper, the ground field $\bbk$ is of characteristic zero.
 
\vskip1.5ex

\noindent 
{\bf Acknowledgement.} I am grateful to Mamuka Jibladze for inspiring conversations.

\section{Seaweed subalgebras of Kac--Moody algebras } \label{KM}

Let $\g$ be a simple finite-dimensional non-Abelian Lie algebra.
We fix a Borel subalgebra $\gt b\subset \gt g$ and a Cartan subalgebra $\gt t\subset\gt b$. 
Set $r=\rk\gt g=\dim\gt t$. 
Let $\Pi=\{\alpha_1,\ldots,\alpha_r\}$ be the set of simple roots associated with 
$(\gt b,\gt t)$ and $\lambda\in\gt t^*$ the highest root.  

Let $\wg$ be an affine Kac--Moody algebra with a set of simple roots $\widehat{\Pi}=\Pi\sqcup\{\alpha_0\}$.
Then 
\begin{equation} \label{KM-sum}
\wg=\gt g[t,t^{-1}]\oplus \bbk C\oplus\bbk d,
\end{equation} 
where $C$ is a central element and $[d,\xi t^k]=k\xi t^k$.   Other commutator relations in $\wg$ are described in \cite[\S\,7.2]{kac}.
By the construction, $\gt h=\gt t\oplus \bbk C\oplus\bbk d$ is a Cartan subalgebra of $\wg$.

For a root $\alpha$ of $(\gt g,\gt t)$, let $e_\alpha\in\gt g$ be a root vector. 
Set $\widehat{\gt b}=(\gt b+\gt h)\oplus t\g[t]$ and $e_{-\alpha_0}=e_{\lambda} t^{-1}\in\g[t^{-1}]$.  
To a subset $S\subset \widehat{\Pi}$, one associates a standard parabolic $\gt p=\gt p(S)\subset\wg$,
which is generated by $\wb$ and $\{e_{-\beta}\mid \beta\in S\}$.  If $S=\widehat{\Pi}$, then clearly
$\gt p=\wg$. 

For convenience of the reader,  we give a description of $\gt p(S)$ with a proper subset $S$. 
Let $\gt p(S\cap \Pi)\subset\gt g$ be the standard parabolic associated with 
$S\cap\Pi$ and $\gt n(S\cap \Pi)$ the nilpotent radical of $\gt p(S\cap \Pi)$. Then let 
$\gt z\subset \gt n(S\cap \Pi)$ be the centre of $\gt n(S\cap \Pi)$. 

\begin{lm} \label{p}
Suppose that $S\ne \widehat{\Pi}$. \\[.2ex]
{\sf (i)} If $\alpha_0\not\in S$, then 
$\gt p(S)=t\gt g[t]+\gt h+\gt p(S\cap \Pi)$.    \\[.2ex]
{\sf (ii)} If $\ap_0\in S$, then $\gt p(S)=t\gt g[t]+\gt h+(\gt p(S\cap \Pi)\ltimes \gt z t^{-1})$, where  $\gt z t^{-1}$
is an Abelian ideal of $\gt p(S\cap \Pi)\ltimes \gt z t^{-1}$.
\end{lm}
\begin{proof}
The positive part $t\gt g[t]$ of $\wg$ is contained in $\widehat{\gt b}$ and hence in any standard parabolic subalgebra. 
Since $\ap_0(d)=1$ and  $[d,\gt g]=0$, we have 
$\gt p(S)\cap \gt g=\gt p(S\cap\Pi)$. Part {\sf (i)} is now clear. 

Suppose $\ap_0\in S$. Then $\Pi\not\subset S$. Standard  facts from the theory of simple Lie algebras are that 
$\lambda=\sum_{i=1}^r n_i \ap_i$ with $n_i\ge 1$ and that
$\beta= \sum_{i=1}^r c_i \ap_i$ with $0\le c_i\le n_i$ for any positive root $\beta$ of $(\gt g,\gt t)$. 
Therefore no combination 
\[
-k\ap_0+\left(\sum_{i=1}^{r} d_i \ap_i\right)-\sum_{\ap_i\in S\cap\Pi} k_i \ap_i
\]
with $k\ge 2$ and $d_i, k_i\ge 0$ is an $\gt h$-weight of $\wg$. Hence 
$\gt p(S)\cap t^{-1}\gt g[t^{-1}]=:\gt p_{-1}$ is contained in $\gt g t^{-1}$. Since $\gt p(S)$ is a Lie 
subalgebra, we have $[\gt p(S\cap \Pi),\gt p_{-1}]\subset\gt p_{-1}$ and 
$[\gt p_{-1},\gt p_{-1}]=0$. 

We have $\gt p_{-1}=Vt^{-1}$, where $V$ is the sum of $\gt g_\beta$ with 
$\beta\in U:=\lambda +\langle S\cap\Pi\rangle_{\BZ}$. Hence $V$ is contained in the nilpotent radical $\gt n(S\cap\Pi)$. 
For any $\beta\in U$ and any $\ap_i\not\in S$, the sum $\beta+\ap_i$ is not a root of $\gt g$. Hence
$[e_{\ap_i},V]=0$. This is true for all elements of $\gt n(S\cap \Pi)$. Hence $V\subset\gt z$.

Let $\gt l(S\cap\Pi)\subset\gt p(S\cap\Pi)$ be the standard Levi subalgebra. 
Note that $V$ is a simple $\gt l(S\cap\Pi)$-module with the highest weight $\lambda$. 
Clearly $[\gt l(S\cap\Pi),\gt z]\subset\gt z$. 

Assume that $V_1\subset\gt z$ is a non-zero simple 
$\gt l(S\cap\Pi)$-module with a highest weight vector $e_\gamma$, where $\gamma\ne\lambda$.  
Since $\gamma\ne\lambda$, where is $\ap_i\in\Pi$ such that 
$[e_{\ap_i},e_\gamma]\ne 0$. If $\ap_i\in S$, then this contradicts the assumption that $\gamma$ is a
highest weight for  $\gt l(S\cap\Pi)$. If $\ap_i\not\in S$, then $e_{\ap_i}\in\gt n(S\cap\Pi)$ and 
$[e_{\ap_i},V_1]=0$, because $V_1\subset\gt z$. Thus $V=\gt z$. 
\end{proof}

Let $\omega$ be an involution of $\wg$ such that %
$\omega|_{\gt h}=-\id_{\gt h}$ and $\omega(\gt g t^k)=\gt g t^{-k}$ for all $k\in\mathbb Z$. 
Then the opposite parabolic $\gt p^-$ is defined by $\gt p^-=\omega(\gt p)$. 
Having two standard parabolics $\gt p=\gt p(S)$ and $\gt r=\gt p(S')$, we define a seaweed subalgebra 
$\wq=\gt p\cap\gt r^-=:\wq(S,S')$.   
The description of  Lemma~\ref{p} implies that $\wq$ is  finite-dimensional if and only if both $\gt p,\gt r$ are smaller than $\wg$, i.e., if 
$S\ne\widehat{\Pi}$ and $S'\ne\widehat{\Pi}$. 
In the following, we always 
assume that both subsets $S$ and $S'$ are proper. 
Clearly $C\in \wq$, but $C\not\in[\wq,\wq]$ by \cite[7.17\,\&\,7.20]{jos2}. This implies that 
$\ind\bar\q=\ind\wq-1$ for $\bar\q=\wq/\bbk C$. 

A formula for the index of $\bar\q$ is given in \cite[Sect.\,9.3]{jos2}. 
It is an analogue of the formula for the index of a seaweed in $\gt g$ suggested by Tauvel and Yu~\cite[Conj.\,4.7]{ty-AIF}
and proven initially in  \cite[Sect.\,8]{jos}. 

Let $\cK(\el(S))=:\cK(S)\subset\gt h^*$ be the cascade of strongly 
orthogonal roots (=\,{\it Kostant's cascade}) in the Levi subalgebra $\el(S)\subset\gt p(S)$, see \cite{jos77,ty-AIF} for the 
details. In particular, $\cK(\g)=\cK(\Pi)\subset\gt t^*$ is the cascade in  $\g$.
Set $\hat{\te}=\gt t\oplus\bbk d$. 
Let $\gt t_{\BR}\subset\gt t$ be the standard real form, i.e., $\ad(h)$ has real eigenvalues for each $h\in\gt t_{\BR}$. 
Then set  $\hat{\te}_{\BR}=\te_{\BR}\oplus\BR d$. Since each root of $\wg$ is zero on $C$, we may safely assume that 
$\cK(S)\subset\hat{\te}^*_{\BR}$. Then 
let $E_S$ be the $\BR$-linear span of $\cK(S)$ in
$\hat{\te}^*_{\BR}$. We have $\dim E_S=\# \cK(S)$, since the elements of a cascade are linearly independent. The TYJ formula reads now: 
\beq    \label{eq:tau-yu}
   \ind\bar{\q}(S,S')=|\widehat{\Pi}| + \dim E_S +\dim E_{S'}  -2\dim (E_S+E_{S'}) ,
\eeq
see~\cite[Conj.\,4.7]{ty-AIF}, \cite[Sect.\,8]{jos}, \cite[Sect.\,9.3]{jos2}, and also 
\cite[Eq.\,(4${\cdot}$2)]{mD}. 
If $S=S_1\sqcup S_2$, where the subsets $S_1$ and $S_2$ are orthogonal, then clearly 
$\cK(S)=\cK(S_1)\sqcup\cK(S_2)$.
 
We record the data
on the cascades in $\gln$ and $\gt{sp}_{2n}$. Let $\esi_i\in\gln^*$ with $1\le i\le n$ be the standard linear functions such that 
$\esi_i(A)=a_{ii}$ for a matrix $A=(a_{ij})$. We regard each $\esi_i$ also as a linear function on the 
standard Cartan subalgebra of $\gln$. Similarly for $\gt g=\gt{sp}_{2n}$, let 
$\{\esi_1,\ldots,\esi_n\}\subset\gt t^*$  be the commonly used basis.  
We have
\beq  \label{eq:K-Pi-gl}
       \cK({\gln})=\{\esi_i-\esi_{n+1-i}\mid i=1,\dots, \left\lfloor \frac{n}{2}\right\rfloor\} \,\text{ and } \ \# \cK({\gln})=\left\lfloor \frac{n}{2}\right\rfloor;
\eeq 
as well as  
\beq  \label{eq:K-Pi-sp}
 \cK({\gt{sp}_{2n}})=\{2\esi_i \mid i=1,\dots, n\} \,\text{ and } \ \# \cK({\gt{sp}_{2n}})=n.
\eeq

\subsection{Seaweed subalgebras in  type $\widetilde{{\sf A}}_r$} \label{sA}
Suppose that $\g=\gt{sl}_{r+1}$. 
Let $\wq(S,S')\subset\wg$ be a standard finite-dimensional affine seaweed. 
The extended Dynkin diagram of type $\widetilde{{\sf A}}_r$ 
is a cycle. Therefore we can change the enumeration of simple roots cyclicly. Since $S'$ is a proper subset
of $\widehat{\Pi}$, we may assume that $\ap_0\not\in S'$.   
Let $\gt r^-\subset\gt g$ be the parabolic subalgebra generated by $\gt b^-=\omega(\gt b)$ 
and $\{e_{\beta} \mid \beta\in S'\}$. 

\begin{prop}\label{pA}
Suppose $S'\subset\Pi$. Then 
 a finite-dimensional standard seaweed subalgebra $\wq=\wq(S,S')\subset\wg$   
 is isomorphic  either to $\gt q \oplus\bbk d\oplus\bbk C$ or to a semi-direct product 
$(\gt q \oplus\bbk d\oplus\bbk C)\ltimes \gt z t^{-1}$,  where $\gt q=\gt p(S\cap\Pi)\cap \gt r^-$ is 
a seaweed in $\g$ and $\gt z$ is the centre of $\gt n(S\cap\Pi)$. 
\end{prop}
\begin{proof}
By Lemma~\ref{p}, $\gt p(S')^-=t^{-1}\g[t^{-1}]+\gt h+\gt r^-$ and either 
$\gt p(S)=t\gt g[t]+\gt h+\gt p(S\cap \Pi)$ or 
$\gt p(S)=t\gt g[t]+\gt h+(\gt p(S\cap \Pi)\ltimes \gt z t^{-1})$. 
Since $\wq=\gt p(S)\cap\gt p(S')^-$, the result follows. 
\end{proof} 

\begin{ex}\label{ex-max}
Suppose that $S=\widehat{\Pi}\setminus\{\ap_d\}$ with $1\le d\le r$ and $S'=\Pi$.
Then 
\[
\gt p(S\cap\Pi)=\gt s(\gt{gl}_d\oplus\gt{gl}_{n-d})\ltimes (\bbk^d{\otimes}(\bbk^{n-d})^*)
\]
with $\gt z=\gt n(S\cap\Pi)=\bbk^d{\otimes}(\bbk^{n-d})^*$. By 
Proposition~\ref{pA}, we have 
$$
\bar{\gt q}(S,\Pi)=(\gt s(\gt{gl}_d\oplus\gt{gl}_{n-d})\oplus\bbk d)\ltimes (2\bbk^d{\otimes}(\bbk^{n-d})^*),
$$ 
where 
$2\bbk^d{\otimes}(\bbk^{n-d})^*=\gt z\oplus\gt z t^{-1}$, \,$[d,\gt z]=0$, and $\ad(d)$ acts  on $\gt z t^{-1}$ as  $-\id$. 
\end{ex}

\subsection{Construction of a graph in type $\widetilde{{\sf A}}_r$} \label{GA}
Suppose that $S=\{\alpha_i \mid i\not\in I\}$, $S'=\{\alpha_i \mid i\not\in I'\}$, where $I,I'\subset\{0,1,\ldots,r\}$. We put $n=r+1$ nodes on a 
circle, as vertices of a regular $n$-gon, labelling them consequently clockwise with the cosets $0+n\Z$,
$1+n\Z,\ldots, r+n\Z$. For $a\in\Z$, set $\overline{a}=a+n\Z$. 
Let $o$ be the centre of the circle. 

Suppose that 
$I=\{i_1,\ldots,i_a\}$, where $0\le i_1<i_2<\ldots < i_a\le r$. Then the node 
$\overline{i_1+1}$ is connected by an arc to the node $\overline{i_2}$ on the {\bf outside} of the circle. 
Say that this arc goes from $\overline{i_1+1}$ to $\overline{i_2}$. 
We continue on the outside joining the vertex  $\overline{i_1+2}$ with $\overline{i_2-1}$,  the vertex $\overline{i_1+3}$ with $\overline{i_2-2}$ and so  on. If $i_2-i_1$ is odd, then there is a middle node with no adjacent outside arc.  
For an arc of the graph going from $A$ to $B$, we call  the arc of the circle going from $A$ to $B$ clockwise
its {\it shadow}. 
It is essential that $o$ lies outside the area bounded by an arc and its shadow. 

Then we do the same for the pairs $(i_2,i_3), \ldots, (i_{a-1},i_a), (i_a,i_1)$. 
If $a=1$, then we have just one pair $(i_1,i_1)$. 
Different outside arcs are disjoint.
The same procedure applies to $I'$, but the outside arcs are replaced here with {\bf inside} arcs. 
The rule concerning $o$ remains the same. Figure~\ref{rules} explains this. 
Let 
$\Gamma=\Gamma(S,S')$ be the resulting graph. 

Without loss of generality, we assume that the arcs of $\Gamma$ meet only 
in vertices. The nodes of $\Gamma$ can be labeled by representatives of  the cosets $a+n\Z$.

\begin{figure}
\begin{tabular}{c|c|c|c}
\hline
correct picture & wrong picture &  correct picture  &   wrong picture  \\
for the outside arcs &  for the outside arcs & for the inside arcs  & for the inside arcs \\
in case $|I|=1$ & in case $|I|=1$ & in case $I'=\{j\}$ & in case $I'=\{j\}$ \\ 
\hline
\trimbox{3.3cm 1.5cm 3cm 0cm}{
\begin{tikzpicture}[scale=0.53]   
 \draw[color=blue, thin](0,0) circle (3);
  \filldraw [gray] (0:3) circle (3pt)
       (45:3) circle (3pt) (2*45:3) circle (3pt) (3*45:3) circle (3pt) (4*45:3) circle (3pt)  (5*45:3) circle (3pt)
        (6*45:3) circle (3pt)  (7*45:3) circle (3pt);
          \draw[thick] (3*45:3)  .. controls (-5.5,0) and (-1.7,-5.3) ..  (0.8,-4.7);
             \draw[thick] (0.8,-4.7)  .. controls (4.7,-4) and (5.8,3) ..  (2*45:3);
            \draw[thick]  (4*45:3)  .. controls (-2,-8) and (7,-2.5) ..  (45:3);  
                \draw[thick] (5*45:3)  .. controls (-1.5,-4.5) and (4,-4) ..  (0:3);   
                \draw[thick] (6*45:3)  .. controls (1.5,-3.5) and (2,-2.5) ..  (7*45:3); 
                   \node (v0) at (0:0) {$\bullet$};
 \end{tikzpicture}
}
 & \trimbox{0cm 1.5cm 1.7cm -0.2cm}{ \begin{tikzpicture}[scale=0.53]   
 \draw[color=blue, thin](0,0) circle (3);
  \filldraw [gray] (0:3) circle (3pt)
       (45:3) circle (3pt) (2*45:3) circle (3pt) (3*45:3) circle (3pt) (4*45:3) circle (3pt)  (5*45:3) circle (3pt)
        (6*45:3) circle (3pt)  (7*45:3) circle (3pt);
          \draw[color=red, thick] (3*45:3)  .. controls (-2,3) and (-1,3.7) ..  (2*45:3);
            \draw[thick]  (4*45:3)  .. controls (-2,-8) and (7,-2.5) ..  (45:3);  
                \draw[thick] (5*45:3)  .. controls (-1.5,-4.5) and (4,-4) ..  (0:3);   
                \draw[thick] (6*45:3)  .. controls (1.5,-3.5) and (2,-2.5) ..  (7*45:3); 
                   \node (v0) at (0:0) {$\bullet$};
                   \node(v1) at (0.7,3.5) {$i_1{+}1$};
 \end{tikzpicture}}  & 
 \trimbox{1.8cm 0cm 3cm 2cm}{
\begin{tikzpicture}[scale=0.53]   
  \node(v1) at (0.2,3.5) {$j{+}1$};
    \node(v2) at (-2.4,2.5) {$j$};
 \draw[color=blue, thin](0,0) circle (3);
  \filldraw [gray] (0:3) circle (3pt)
       (45:3) circle (3pt) (2*45:3) circle (3pt) (3*45:3) circle (3pt) (4*45:3) circle (3pt)  (5*45:3) circle (3pt)
        (6*45:3) circle (3pt)  (7*45:3) circle (3pt);
          \draw[thick] (3*45:3)  ..  controls (-0.2,-3) and (2.2,-1) ..  (2*45:3);
            \draw[thick]  (4*45:3)  .. controls (-2,-4) and (4,-1) ..  (45:3);  
                \draw[thick] (5*45:3)  .. controls (-1,-3) and (2,-2.5) ..  (0:3);   
                \draw[thick] (6*45:3) -- (7*45:3); 
                   \node (v0) at (0:0) {$\bullet$};
 \end{tikzpicture}} 
 & 
  \trimbox{1.8cm 0cm 3cm 2cm}{
\begin{tikzpicture}[scale=0.53]   
  \node(v1) at (0.2,3.5) {$j{+}1$};
    \node(v2) at (-2.4,2.5) {$j$};
 \draw[color=blue, thin](0,0) circle (3);
  \filldraw [gray] (0:3) circle (3pt)
       (45:3) circle (3pt) (2*45:3) circle (3pt) (3*45:3) circle (3pt) (4*45:3) circle (3pt)  (5*45:3) circle (3pt)
        (6*45:3) circle (3pt)  (7*45:3) circle (3pt);
          \draw[color=red, thick] (3*45:3)  ..  controls (0,0) and (0.5,1) ..  (2*45:3);
            \draw[thick]  (4*45:3)  .. controls (-2,-4) and (4,-1) ..  (45:3);  
                \draw[thick] (5*45:3)  .. controls (-1,-3) and (2,-2.5) ..  (0:3);   
                \draw[thick] (6*45:3) -- (7*45:3); 
                   \node (v0) at (0:0) {$\bullet$};
 \end{tikzpicture}} 
\end{tabular}
\caption{Rules concerning $o$.} \label{rules}
\end{figure}

In Figure~\ref{ex3}, we present three graphs $\Gamma$. For the first one, $n=10$, $I=\{9\}$,  
and $I'=\{4,8\}$. For the second one,
$n=9$, $I=\{3,8\}$, and $I'=\{2,6\}$. Finally for the third one, $n=8$, $I=\{1,5\}$, $I'=\{3,7\}$. 

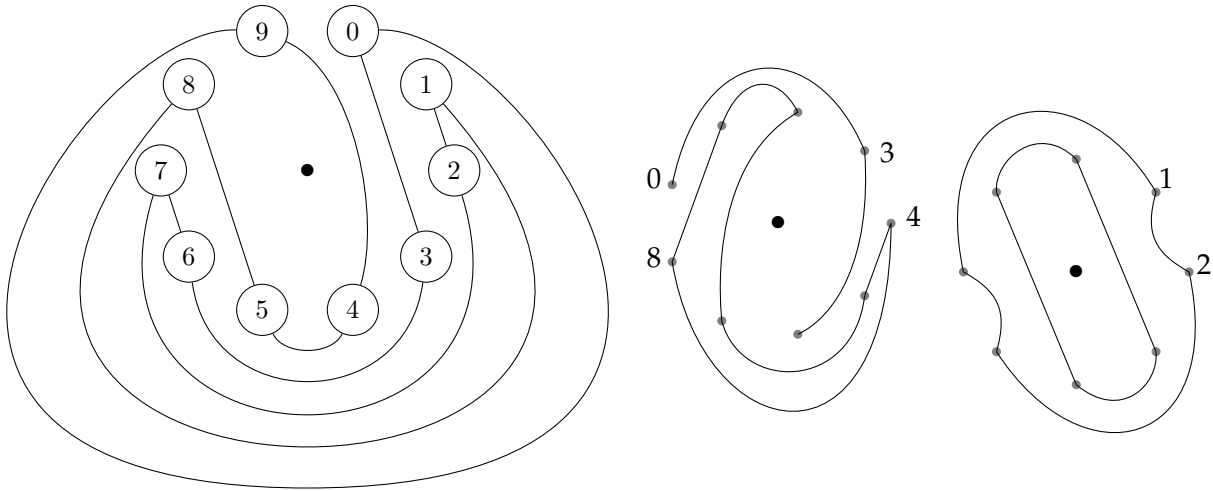
\begin{figure}
  \trimbox{3.5cm 1cm 3cm -0.5cm}{
 \begin{tikzpicture}[scale=0.65]                   
    \node (v0) at (0:0) {$\bullet$};
  \tikzstyle{every node}=[draw,shape=circle];
  \node (v1) at ( 0:3) {{\footnotesize $2$}} ;
  \node (v2) at ( 36:3) {{\footnotesize $1$}} ;
  \node (v3) at (2*36:3) {{\footnotesize $0$}} ;
  \node (v4) at (3*36:3) {{\footnotesize $9$}} ;
  \node (v5) at (4*36:3) {{\footnotesize $8$}} ;
  \node (v7) at (5*36:3) {{\footnotesize $7$}} ;
    \node (v8) at (6*36:3) {{\footnotesize $6$}} ;
      \node (v9) at (7*36:3) {{\footnotesize $5$}} ;
        \node (v10) at (8*36:3) {{\footnotesize $4$}} ;
          \node (v6) at (9*36:3) {{\footnotesize $3$}} ;
          \draw (v1) -- (v2) ;
           \draw (v4) .. controls   (1,2)  and (1.5,-1)  ..  (v10) ;
             \draw (v5) -- (v9) ; 	
               \draw (v3) -- (v6) ; 
                 \draw (v8) -- (v7) ; 
       \draw (v4) .. controls   (-4.5,3)  and (-11,-6.5)  ..  (0,-6.5) ;     
            \draw (0,-6.5) .. controls   (11,-6.5)  and (4.5,3)  ..  (v3) ;            
           \draw (v5) .. controls (-11,-8) and (11,-8) ..  (v2) ;       
         \draw (v9) .. controls (-0.5,-3.8) and (0.5,-3.8) ..  (v10) ;            
       \draw (v8) .. controls (-2,-5) and (2,-5) ..  (v6) ;           
       \draw (v7) .. controls (-5,-6.5) and (5,-6.5) ..  (v1) ;            
 \end{tikzpicture}}
  \trimbox{0cm 0cm 0.5cm 0cm}{
\begin{tikzpicture}[scale=0.5]   
  \filldraw [gray] (0:3) circle (3pt)
       (40:3) circle (3pt) (2*40:3) circle (3pt) (3*40:3) circle (3pt) (4*40:3) circle (3pt)  (5*40:3) circle (3pt)
        (6*40:3) circle (3pt)  (7*40:3) circle (3pt)  (8*40:3) circle (3pt);
        \draw  (8*40:3) -- (0:3) ; 
          \draw  (7*40:3)  .. controls (1.5,-2.5) and (2.5,-1) ..  (40:3) ; 
           \draw  (6*40:3) .. controls (-1.5,-2.5) and (-2,1.5).. (2*40:3) ; 
            \draw  (5*40:3) -- (3*40:3) ; 
          \draw (3*40:3)  .. controls (-1,4) and (0,4) ..  (2*40:3);
            \draw (4*40:3)  .. controls (-2,5) and (1,5) ..  (40:3);  
                \draw (5*40:3)  .. controls (-2,-6) and (3,-7) ..  (0:3);  
                \draw (6*40:3)  .. controls (-1,-4.5) and (2,-4.5) ..  (8*40:3);  
              \node (v0) at (0:0) {$\bullet$};    
              \node (v1) at (-3.3,1.2) {{\small 0}};
                 \node (v3) at (2.9,1.9) {{\small 3}};
                  \node (v4) at (3.6,0.2) {{\small 4}};
                     \node (v8) at (-3.3,-0.9) {{\small 8}};
 \end{tikzpicture} }  \ \ 
 \begin{tikzpicture}[scale=0.5]   
  \filldraw [gray] (0:3) circle (3pt)
       (45:3) circle (3pt) (2*45:3) circle (3pt) (3*45:3) circle (3pt) (4*45:3) circle (3pt)  (5*45:3) circle (3pt)
        (6*45:3) circle (3pt)  (7*45:3) circle (3pt);
          \draw  (45:3) .. controls (1.8,1) and (2.1,0.5) ..   (0:3) ; 
           \draw  (2*45:3) -- (7*45:3) ; 
            \draw  (3*45:3) -- (6*45:3) ; 
                 \draw  (4*45:3) .. controls (-2.1,-0.5)  and   (-1.8,-1) ..  (5*45:3) ; 
          \draw (3*45:3)  .. controls (-2.2,3) and (-1,4) ..  (2*45:3);
            \draw (4*45:3)  .. controls (-4,4.7) and (0,5.7) ..  (45:3);  
                \draw (5*45:3)  .. controls (0.2,-5.7) and (4,-4.7) ..  (0:3);   
                 \draw (6*45:3)  .. controls (1.2,-4) and (2.2,-3) ..  (7*45:3); 
                   \node (v0) at (0:0) {$\bullet$};
          \node (v1) at (2.4,2.5)       {{\small 1}};   
              \node (v2) at (3.4,0.2)       {{\small 2}};   
 \end{tikzpicture} 
\caption{Examples of graphs $\Gamma$.} \label{ex3}
\end{figure}

Each node of $\Gamma$ is a vertex of at most $2$ arcs.
Therefore a connected component of  $\Gamma$ is either a segment or a cycle. 
There are no loops in $\Gamma$. 
But note that $\Gamma$ does not have to be a simple graph. It may contain a cycle 
with two vertices $i,j$ and two arcs $(i,j)_{\sf outs}$, $(i,j)_{\sf ins}$, where the first one is an outside and the second an inside arc. In that case, ``an arc $(i,j)$ of $\Gamma$" refers to any of the two.

\section{Combinatorial interpretation of the index in type $\widetilde{{\sf A}}_r$} \label{d}

In this section, $\gt g=\gt{sl}_{n}$.  
We choose $\{1,\ldots,n\} \subset\Z$ as a set of representatives for $\Z/n\Z$. 
Now the vertices of $\Gamma$ are numbered from $1$ to $n$. 
To the arc of the circle joining a vertex $i$ with $i{+}1$ clockwise we assign the simple root $\ap_i$.
The arc joining $n$ and $1$ is labeled with $\ap_0$.  

Recall that $\esi_j\in\gt{gl}_n^*$ with $1\le j\le n$ are the standard linear functions.  
We regard them also as functions on $\te^*$ and on $\gt h$ by  
setting $\esi_i(d)=\esi_i(C)=0$ for each $i$. 
Of course, $\sum_{i=1}^n \esi_i$ is zero on the Cartan subalgebra $\gt t\subset\gt{sl}_n$. 
This circumstance makes several arguments more involved. 

Set $\delta=\sum_{i=0}^r \ap_i \in\gt h^*$,  where $r=n-1$.  Then $\delta$ is an imaginary root such that  $\delta|_{\gt t}=0$, 
$\delta(d)=1$, and $\delta(C)=0$, while $\ap_0=\esi_n-\esi_1+\delta$ and $\ap_i=\esi_{i}-\esi_{i+1}$ for $i>0$.  In this terms, the label of an arc $(i,i+1)$ is 
$\esi_{i}-\esi_{i+1}$ for $1\le i\le r$.

There is a bijections between the outside arcs of $\Gamma$ and $\cK(S)$. Namely, an arc $(i,j)$ going from a vertex $i$ to 
$j$ corresponds to 
$$
\sum_{u\in[i,j-1]}\alpha_u, \
\text{ where } 
\ [i,j-1]=\{i,i+1,\ldots,r,0,1,\ldots,j-1\} \ \text{ if } j<i. 
$$
The sum is equal to $\esi_i-\esi_j$ if $j>i$ and to $\esi_i-\esi_j+\delta$, if $j<i$, cf.~\eqref{eq:K-Pi-gl}. 
The similar bijection exists between $\cK(S')$ and the inside arcs.
Uniformly, we assign to an arc of $\Gamma$ the sum of simple roots belonging to its shadow. 

Let $(i,j)$ be an arc of $\Gamma$, which may go from $i$ to $j$ or from $j$ to $i$. Let 
$\beta(i,j)\in \cK(S)\cup \cK(S')$ be the element assigned to it. Then 
\[
\beta(i,j)\in\{\esi_i-\esi_j, \esi_j-\esi_i,  \esi_i-\esi_j+\delta, \esi_j-\esi_i+\delta \}.
\]

\begin{lm} \label{ind}
Let $B\subset \Gamma$ be a connected subgraph consisting of the  pairwise distinct vertices 
$j_1,\ldots,j_m$ and the edges $(j_i,j_{i+1})$ with $1\le i<m$. Then 
the set 
$$
\bar K(B):=\{ \beta(j_i,j_{i+1})|_{\gt t}  \mid 1\le i<m\} \subset\gt t^*
$$
is linearly independent and $|\bar K(B)|=m-1$.
\end{lm}
\begin{proof}
Suppose that $X=\sum_{i=1}^{m-1} c_i \bar\beta(j_i,j_{i+1})=0$ for some $c_i\in\bbk$. 
We have 
\[
X=\pm c_1 \esi_{j_1}+ \left(\sum_{i=2}^{m-1} (\pm c_{i-1}\pm c_i) \esi_{j_i}\right) \pm c_{m-1} \esi_{j_m}.
\]
Each difference $\esi_i-\esi_j$ is zero on the centre 
$\bbk I_n\subset\gt{gl}_n$. Hence 
$X$, regarded an an element of $\gt{gl}_n^*$, is zero on $\bbk I_n$. Thereby $X$ is zero on $\gt{gl}_n$ and $c_1=c_{m-1}=0$. Furthermore,
$c_{i-1}=\pm c_{i}$ for $2\le i \le m-1$. Thus, $c_i=0$ for each $i$. 
\end{proof}

For a connected component $B$ of  $\Gamma$,
let $\cv(B)$ be the set of vertices of $B$ and $\cE(B)$ the set of edges. 
Set $K(B)=\{ \beta(i,j) \mid (i,i)\in \cE(B)\}\subset\gt h^*$ and 
$$
\bar K(B)= \{ \beta(i,j)|_{\gt t} \mid (i,i)\in \cE(B)\}\subset\gt t^*.
$$  

\begin{lm} \label{round}
Let $B$ be a cycle in $\Gamma$ that does not contain $o$ in its interior. 
Let $j_1,\ldots,j_m$ be the consecutive vertices of $B$.
For each $\e\in\cE(B)$ joining a vertex $j_i$ with $j_{i+1}$ or $j_m$ with $j_1$,
let $c_{\e}\in\{1,-1\}$ be a coefficient defined by the following rule. If the shadow 
of $\e$ goes from $j_i$ to $j_{i+1}$ (or from  $j_m$ to $j_1$) clockwise, then $c_{\e}=1$, 
otherwise $c_{\e}=-1$. Then 
$Y(B):=\sum\limits_{\e\in\cE(B)} c_{\e} \beta(\e)=0$.
\end{lm}
\begin{proof}
Consider an arrow $\overrightarrow{og}$ starting at $o$ and ending at a point $g$ of $B$. Let $g$ move along $B$ making one round trip. The arrow is moving accordingly, rotating around $o$. Since $o$ does not lie in the interior of $B$, the 
total angle of the rotation is zero. We extend $\overrightarrow{og}$ to a half-line $L_{\overrightarrow{og}}$ and let
$\ell(g)$ be the intersection point of $L_{\overrightarrow{og}}$ and the circle. Then the total distance, counted with a sign,
traveled by $\ell(g)$ on the circle is zero.   The equality $Y(B)=0$ is an analogue of this fact. 
 
Assume that $g=j_1$ at first and that it goes to $j_2$, $j_3$ and so on until $j_m$ and then back to $j_1$.  When 
$g$ moves from $u$ to $v$ along an edge $(u,v)\in \cE(B)$, the point $\ell(g)$ moves from $u$ to $v$ along the shadow of $(u,v)$.  For each arc $(\overline{i},\overline{i+1})$ of the circle, $\ell(g)$ travels from  
$\overline{i}$ to $\overline{i+1}$ clockwise as many times as from $\overline{i+1}$ to $\overline{i}$ anti-clockwise, because  
the total rotation angle is zero. Hence each $\ap_j\in\widehat{\Pi}$  appears in  $Y(B)$ with the zero coefficient and 
$Y(B)=0$.
\end{proof}

\begin{thm} \label{thm-A}
Let $\wq=\wq(S,S')$ be a  standard finite-dimensional seaweed subalgebra of $\wg$ for $\gt g=\gt{sl}_n$.
Set $\bar\q=\wq/\bbk C$. 
 Then 
$$
\ind\bar\q= 2\#\{\text{the cycles of}\ \Gamma\} + \#\{\text{the segments of $\Gamma$}\}-\iota$$ 
for $\Gamma=\Gamma(S,S')$, where $\iota=0$ if $\Gamma$ has no cycles with $o$ in the interior and $\iota=2$ otherwise. 
\end{thm}
\begin{proof}
We use~\eqref{eq:tau-yu}. 
The relation between the arcs of $\Gamma$ and the elements of $\cK(S)$ and $\cK(S')$ shows that 
$\dim E_S +\dim E_{S'}$ is equal to the number of arcs in $\Gamma$, say $a$.  Since each node of $\Gamma$ is a vertex of at most $2$ arcs, we have $a=n-\#\{\text{the segments of $\Gamma$}\}$. 

In order to compute $\dim(E_S+E_{S'})$, we 
describe relations among the elements of $\cK(S,S'):=\cK(S)\cup-\cK(S')$.
The minus sign in this formula is used because $\cK(S)\cap\cK(S')$ does not have to be empty. 

For each connected component $B$ of $\Gamma$, we have 
$\bar K(B)\subset \langle \esi_i \mid i \in \cv(B)\rangle_{\BR}$. 
If $B_1$ and $B_2$ are distinct components, then $\cv(B_1)\cap\cv(B_2)=\varnothing$. 
Since each difference $\esi_i-\esi_j$ is zero on $\bbk I_n$, a minimal relation among elements 
of $\cK(S,S')$ is (up to signs) a non-trivial linear dependence among 
$\bar\beta(\e)=\beta(\e)|_{\gt t}$ with $\e\in \cE(B)$ for some
$B$. By Lemma~\ref{ind}, there are no such dependences if
 $B$ is a segment and  at most one  if $B$ is a cycle. 

Let $B$ be a cycle with the  pairwise distinct vertices 
$j_1,\ldots,j_m$ and the edges $(j_i,j_{i+1})$ with $1\le i\le m$, where 
$(j_m,j_{m+1})$ is an arc connecting $j_m$ with $j_1$. Then 
\begin{equation} \label{B}
\sum\limits_{(u,v)\in\cE(B)} c_{(u,v)} \bar\beta(u,v) =0
\end{equation} 
for  
$c_{(u,v)}\in \{1,-1\}$. Thus,
\[
\dim((E_S+E_{S'})|_{\gt t}) = \dim E_S + \dim E_{S'} - \#\{\text{the cycles of}\ \Gamma\}. 
\]
  
It remains to compare $\dim(E_S+E_{S'})$ and 
$\dim((E_S+E_{S'})|_{\gt t})$. These dimensions are equal if $\delta\not\in E_S+E_{S'}$ and differ by $1$ otherwise. 
For each cycle $B$ of $\Gamma$, set $X(B)=\sum\limits_{\e\in\cE(B)} c_{\e} \beta(\e)$, where
the coefficients $c_{\e}$ are the same as in~\eqref{B}. Then $X(B)\in \BR\delta$. 
Let $F(B)$ be  the face of $\Gamma$ bounded by $B$.
Now we show that $X(B)\ne 0$ if and only if  $o\in F(B)$. 

Let $C$ be the half-line starting at $o$ and going through the middle of the circle arc $(n,1)$. 
Without loss of generality, assume that $C\cap\Gamma$ is a finite set and that $C$ intersects each arc of $\Gamma$ at most once and transversally. 
Then $C$ intersects the shadow of an arc $(i,j)$ of $\Gamma$ if and only if it intersects $(i,j)$. 
Note that $\beta(i,j)=\delta\pm(\esi_i -\esi_j)$ if and only if $C$ intersects the shadow of $(i,j)$.    

Suppose first that $o\in F(B)$. Since $C$ intersects $B$ odd number of times, the number of $(u,v)\in\cE(B)$ such that 
$C\cap (u,v)\ne\varnothing$  is odd. Hence $X(B) = k \delta$ for $k\in 1+2\Z$.

Suppose now that  
$o\not\in F(B)$. 
%
Then 
Lemma~\ref{round} implies that 
$X(B)=0$. 
This can be seen in another way, if we define {\it negative} and {\it positive} crossings of $C$ and $B$.
Roughly speaking, a  crossing is positive, if $C$ enters $F(B)$ and negative otherwise. 
Depending on this,  $\beta(i,j)=\delta\pm(\esi_i -\esi_j)$ appears in $X(B)$ either with coefficient $1$ or $-1$.
Since the numbers of negative and positive crossings coincide, the coefficient of $\delta$ in $X(B)$ is zero. 

Recall that  a minimal relation among the elements 
of $\cK(S,S')|_{\gt t}$ is of the form~\eqref{B} 
for some cycle $B$. Thus,
$\delta\in E_S+E_{S'}$ if and only if $\delta\in \langle K(B)\rangle_{\BR}$ for a cycle $B$. 
Now we have $\dim(E_S+E_{S'})=\dim((E_S+E_{S'})|_{\gt t}) + \frac{1}{2}\iota$. Rewriting~\eqref{eq:tau-yu} brings 
\begin{align*}
& \ind\bar\q=n+\dim E_S +\dim E_{S'}-2\left(\dim((E_S+E_{S'})|_{\gt t}) + \frac{1}{2}\iota\right)= \\
 & =
  n+a-2\left(\dim E_S +\dim E_{S'} - \#\{\text{the cycles of}\ \Gamma\}+ \frac{1}{2}\iota\right)= \\
& =   n - a + 2\#\{\text{the cycles of}\ \Gamma\}  -\iota = \#\{\text{the segments of $\Gamma$}\} + 2\#\{\text{the cycles of}\ \Gamma\}  -\iota.
\end{align*}
This finishes the proof. 
\end{proof}

\begin{rmk}\label{iota}
Our proof of Theorem~\ref{thm-A} shows that there is another recipe for computing $\iota$. 
We call an arc $(i,j)$ of $\Gamma$ {\it affine} if its shadow contains $\ap_0$, i.e., if 
$\beta(i,j)=\delta\pm(\esi_i-\esi_j)$. Then $o\in F(B)$ if and only if $\cE(B)$ contains an odd number of affine arcs. 
\end{rmk}

\begin{ex} \label{3}
Consider the graphs presented in Figure~\ref{ex3} and corresponding seaweeds. The first graph is a cycle and $o$ lies in its interior. This $\Gamma$ has $3$ affine arcs, namely $(0,9)$, $(0,3)$, and $(9,4)$. Hence 
$\iota=2$ and $\ind\bar\q=2-2=0$. 

The second graph has one segment and one cycle $B$ with $o\not\in F(B)$. Here $B$ has two affine arcs, $(7,2)$ and $(8,1)$.
We have $\iota=0$ and $\ind\bar\q=1+2=3$.

The third graph consists of two cycles. Both of them have $o$ in the interior. Hence
$\ind\bar\q=4-2=2$.   
\end{ex}

\begin{ex}\label{1}
As an illustration, we consider the case, where $|I|=|I'|=1$. 
Suppose $I=I'$. Then $\bar\q=\gt{sl}_n\oplus \bbk d$.
The graph $\Gamma$ has $\left\lfloor\frac{n}{2}\right\rfloor$ cycles and $n-2\!\left\lfloor\frac{n}{2}\right\rfloor$ segments.
The centre $o$ is not contained in $F(B)$ for any of the cycles. By Theorem~\ref{thm-A}, 
$\ind\bar\q=2\!\left\lfloor\frac{n}{2}\right\rfloor + n-2\!\left\lfloor\frac{n}{2}\right\rfloor=n=1+\ind\gt{sl}_n$. 

Suppose now that $I=\{0\}$ and $I'=\{d\}$ with $1\le d \le n/2$. 
The corresponding seaweed subalgebra is described in Example~\ref{ex-max}.
Consider two reflections $s,s'$ that are symmetries of the 
$n$-gon, where  $s(1)=n$ and $s'(d)=d+1$. 
  By the construction, two nodes of $\Gamma$ lie in the same connected component if and only if they lie in one and the same orbit of the group ${\tt H}=\left<s,s'\right>$. 
The composition  $s'{\circ}s$ 
is the rotation 
$i\mapsto i+2d \ (\!\!\!\!\mod n)$. The order of $s'{\circ}s$ is $\frac{n}{\gcd(n,2d)}$. 
It is easier to describe  the connected component of $\Gamma$ for even and odd $n$ separately. 

Suppose $n=2k$. Then $\Gamma$ has no segments and each cycle in $\Gamma$ has  
$|{\tt H}|=\frac{2k}{\gcd(k,d)}$ nodes. The number of cycles is therefore $\gcd(k,d)$.  
We obtain, $2{\cdot}\#\{\text{the cycles of}\ \Gamma\}=2\gcd(k,d)=\gcd(2k,2d)=\gcd(n,2d)$. 

Consider now $n=2k+1$. Here $\Gamma$ has one segment.
The segment contains $\frac{1}{2}|{\tt H}|=\frac{n}{\gcd(n,2d)}$ nodes. 
Each cycle in $\Gamma$ has  $|{\tt H}|$ nodes.
Hence there are 
$$
\left(n-\frac{n}{\gcd(n,2d)}\right)\frac{\gcd(n,2d)}{2n}=\frac{1}{2}(\gcd(n,2d)-1) 
$$
cycles in $\Gamma$. Thus, $2{\cdot}\#\{\text{the cycles of}\ \Gamma\} +
\#\{\text{the segments of $\Gamma$}\}=\gcd(n,2d)$. 

Now we determine the value of $\iota$ following the idea of Remark~\ref{iota}. 
There are $d$ affine arcs in $\Gamma$. All of them are inside arcs. They join 
$d$ with $d{+}1$, the node $d{-}1$ with $d{+}2$, and so on until the arc $(1,2d)$. Note that $2d\le n$. 
For a cycle $B$, the number of affine arcs in $B$ is equal to 
$|\cv(B)\cap\{1,2,\ldots,d\}|$. 

If $n$ is even and $d$ is odd, then there is at least one cycle with an odd number of affine arcs. Thus 
here $\iota=2$.  

Suppose that $n$ is odd. 
Set $a=\gcd(n,2d)=\gcd(n,d)$. Then $a|d$. For each cycle $B$, we have 
$|\cv(B)\cap\{1,2,\ldots,a\}|=2$ and hence $|\cv(B)\cap\{1,2,\ldots,d\}|=2\frac{d}{a}$ is even. Here $\iota=0$. 

The last case is $n,d\in 2\Z$. If $a=\gcd(n,2d)$ divides $d$, then again $|\cv(B)\cap\{1,2,\ldots,d\}|=2\frac{d}{a}$
for each cycle $B$ and $\iota=0$. Suppose that $a$ does not divide $d$. 
Since $a|2d$, we have $d=ca+\frac{a}{2}$ with $c\in\Z$. 
Here $|\cv(B)\cap\{1,2,\ldots,d\}|=2c+1$ for each cycle $B$ and $\iota=2$.

The uniform answer is: $\iota=0$ if $\gcd(n,2d)$ divides $d$ and $\iota=2$ otherwise. We have 
$\ind\bar\q=\gcd(n,2d)-\iota$. 
\end{ex}
 
\section{Affine meander graphs in type $\widetilde{{\sf C}}_r$} \label{secC}

In this section, $\g=\gt{sp}_{n}$ with $n=2r$. 
To proper subsets $S,S'\subset \widehat{\Pi}$ we associate a graph $\Gamma^{\sf C}(S,S')$. 
Similar to type {\sf A}, we put $n$ nodes on a 
circle, as vertices of a regular $n$-gon, labelling them consequently clockwise with the numbers from $1$ to $n$.  
The arcs of the circle are labeled with the roots $\ap_i\in \widehat{\Pi}$. 
To the arc $(n,1)$ we attach $\ap_0$ and to $(r,r+1)$ the root $\ap_r$. Each of the arcs
$(i,i+1)$ and $(2r-i,2r+1-i)$ is labeled with $\ap_i$ for $1\le i <r$. We illustrate this procedure in Figure~\ref{arcsC}. 
Let $o$ be the centre of the circle. 

\begin{figure}
 \trimbox{1cm 0.3cm 0cm 0cm}{
 \begin{tikzpicture}[scale=0.7]                   
  \draw (0,0) circle (3);
   \filldraw [gray] (0:3) circle (3pt)
( 36:3) circle (3pt) (2*36:3) circle (3pt)   (3*36:3) circle (3pt) 
 (4*36:3) circle (3pt)
   (5*36:3) circle (3pt)
    (6*36:3) circle (3pt)
       (7*36:3) circle (3pt)      circle (3pt) (8*36:3) circle (3pt)
       (9*36:3)    circle (3pt); 
\node(v1) at (-0.9,-3.3) {{\small $1$}};        \node(v10) at (1.1,-3.3) {{\small $n$}};   
\node(v2) at (-2.8,-1.8) {{\small $2$}};        \node(v9) at (3.4,-1.8) {{\small $n{-}1$}};   
\node(v3) at (-3.3,1.8) {{\small $r-1$}};        \node(v8) at (3.4,1.8) {{\small $r{+}2$}};   
\node(v4) at (-1,3.2) {{\small $r$}};        \node(v7) at (1.6,3.2) {{\small $r{+}1$}};   
\node(v5) at (0.1,3.3) {$\ap_r$}; 
\node(v6) at (-2.4,2.6) {$\ap_{r-1}$}; \node(v11) at (2.7,2.4) {$\ap_{r-1}$}; 
\node(v12) at (0.1,-3.3) {$\ap_0$}; 
\node(v13) at (-2,-2.7) {$\ap_{1}$}; \node(v14) at (2.1,-2.7) {$\ap_{1}$}; 
\node(v16) at (-3.2,-1) {$\ap_{2}$}; \node(v15) at (3.3,-1) {$\ap_{2}$}; 
 \end{tikzpicture}}
\caption{Arcs and simple roots in type {\sf C}.}  \label{arcsC}
\end{figure}
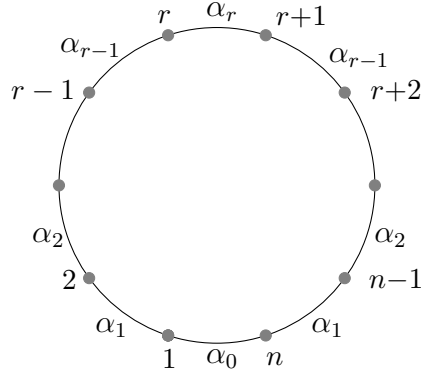

\subsection{Construction of the graph $\Gc(S,S')$} \label{GraphC}
First we remove the arcs labelled with $\ap_i$ such that $\ap_i\not\in S$. 
The circle becomes a union of connected components.  
Let $j_1,\ldots,j_m$ be the consecutive vertices in a connected component $C_k$. For each $1\le i\le m/2$, we draw an arc
$(j_i,j_{m+1-i})$
of $\Gamma^{\sf C}=\Gamma^{\sf C}(S,S')$, connecting the vertex $j_i$ with $j_{m+1-i}$. Each arc   $(j_i,j_{m+1-i})$ lies outside of the circle. 
The shadow of $(j_i,j_{m+1-i})$ is a segment of $C_k$ connecting $j_i$ with $j_{m+1-i}$. 
The construction applies to all connected components. The same procedure is repeated for $S'$, but the outside arcs are replaced here with inside arcs. 
We require that 
$o$ lies outside the area bounded by any arc of $\Gamma^{\sf C}$ and its shadow. 
The arcs are drawn in such a way that the graph has no self-intersections. 

\subsection{Properties of the graph $\Gc(S,S')$} \label{PropGraphC}
Each meander graph of type $\widetilde{{\sf C}}_r$ is also a meander graph of type 
$\widetilde{{\sf A}}_{2r-1}$.   Each meander graph of type 
$\widetilde{{\sf A}}_{2r-1}$ that is symmetric with respect to a reflection with no fixed vertices is 
a  meander graph of type $\widetilde{{\sf C}}_r$. We keep notation of Sections~\ref{GA} and \ref{d}. 

Let $\sigma\!:\BR^2\to\BR^2$ be the reflection such that $\sigma(1)=n$ for the vertices $1$ and $n$ on the circle. 
Without loss of generality, assume that $\Gamma^{\sf C}$ is $\sigma$-stable. 
Let $L$ be the axis of $\sigma$ and $\langle\sigma\rangle\subset{\rm O}(\BR^2)$ the subgroup of order $2$ generated by $\sigma$.

In the following lemma, we collect several observation on the structure of $\Gamma^{\sf C}$.

\begin{lm} \label{lmC}
{\sf (i)} Each connected component of\/ $\Gamma^{\sf C}$ is either a segment or a cycle. \\[.2ex]
{\sf (ii)} Let $(i,j)$ be an arc of\/ $\Gamma^{\sf C}$. Then either $(i,j)$ is $\sigma$-stable and
$i+j=n+1$ or $(i,j)\cap L =\varnothing$. \\[.2ex]
{\sf (iii)} Let $B$ be a connected component of $\Gamma^{\sf C}$. Then either $B$ is $\sigma$-stable  or $B\cap L =\varnothing$. \\[.2ex]
{\sf (iv)} Let $B$ be a $\sigma$-stable connected component of $\Gamma^{\sf C}$. If $B$ is a segment, then exactly one arc of $B$ intersects $L$. If $B$ is a cycle, then exactly two different arcs of $B$  intersect $L$.
\end{lm}
\begin{proof}
Statements {\sf (i)} and {\sf (ii)} are direct consequences of the construction rules. 

Suppose $B\cap L\ne\varnothing$. Then there is $(i,j)\in\cE(B)$ such that $(i,j)\cap L\ne\varnothing$. By part {\sf (ii)},
the arc $(i,j)$ is $\sigma$-stable. Hence $B\cap \sigma(B)\ne\varnothing$ and $\sigma(B)=B$. This proves part {\sf (iii)}. 

By {\sf (iii)} and {\sf (ii)}, 
a $\sigma$-stable component $B$ contains a $\sigma$-stable arc $(i,j)$. 
Assume that $j>r$. Then $i=2r+1-j\le r$. Let $i_1=i, i_2=j, i_3, \ldots, i_m$ be the maximal sequence of consecutive vertices of $B$ such that 
$i_k>r$ for $2\le k\le m$. If $i_m$ is a leaf (adjacent to only one arc), then $B$ is a segment. Furthermore,
$\cv(B)=\{\sigma(i_m),\ldots,\sigma(i_3), i,j,i_3, \ldots, i_m\}$ and exactly one arc of $B$ intersects $L$. 

Suppose that $i_m$ is not a leaf. If $\cv(B)=\{i,j\}$, then $B$ has exactly two different arcs joining $i$ and $j$.
Both of them intersect $L$. If $|\cv(B)|>2$, then there is $(i_m,i_{m+1})\in\cE(B)$, where  
$i_{m+1}=\sigma(i_m)\le r$ and $i_{m+1}\ne i$. The vertices 
$\{\sigma(i_m),\ldots,\sigma(i_3), i,j,i_3, \ldots, i_m\}$ belong to a cycle, which has to be equal to $B$. We see that only the arcs 
$(i,j)$ and $(i_m,\sigma(i_m))$ intersect $L$. Now part {\sf (iv)} is settled.  
\end{proof}

\subsection{Interpretation of the index} \label{sec-C-ind}
If we remove from the circle the arcs labeled with $\ap_i\not\in S$, then there is a bijection 
between the $\sig$-orbits on the connected components 
and the simple non-Abelian  ideals of
the Levi  subalgebra $\gt l(S)\subset\wg$.  Furthermore, there is
a bijection between  the $\sig$-orbits on the  outside arcs of $\Gc$ and $\cK(S)$. Similarly, 
 there is
a bijection between  the $\sig$-orbits on the inside arcs of $\Gc$ and $\cK(S')$. To each arc $(i,j)$ of $\Gc$ we assign 
$\beta(i,j)\in \cK(S)\cup \cK(S')$, which is 
the sum
of simple roots belonging to the shadow of $(i,j)$. If $(u,v)=\sigma((i,j))$, then 
$\beta(u,v)=\beta(i,j)$. 

Recall that 
in type ${\sf C}_r$, we have a standard basis $\{\esi_1,\ldots,\esi_r\}\subset\gt t^*$. 
Assume that $\esi_i(d)=\esi_i(C)=0$ for each $i$.
Then $\ap_i=\esi_{i}-\esi_{i+1}$ for $1\le i < r$ and $\ap_r=2\esi_r$. Furthermore,
$\ap_0=\delta-2\esi_1$, where 
$\delta=\ap_0+\ap_r+2\sum_{i=1}^{r-1}\ap_i$ is an imaginary root with $\delta(d)=1$.
By the construction, 
\[
\beta(i,n+1-i) \in \{ 2\esi_i, \delta-2\esi_i\}  \ \text{ if } \ 1\le i\le r
\]
and $\beta(i,j)=\esi_i-\esi_j$ for $1\le i<j\le r$, cf.~\eqref{eq:K-Pi-gl},~\eqref{eq:K-Pi-sp}. 

\begin{thm} \label{thmC}
Let $\wq=\wq(S,S')$ be a  standard finite-dimensional seaweed subalgebra of $\wg$ for $\gt g=\gt{sp}_n$.
Set $\bar\q=\wq/\bbk C$. 
 Then 
$$
\ind\bar\q= 1+ \#\{\text{the cycles of}\ \Gamma^{\sf C}\} + \frac{1}{2}\#\{\text{the non-$\sigma$-stable segments of $\Gamma^{\sf C}$}\}-\iota
$$ 
for 
$\Gamma^{\sf C}=\Gamma^{\sf C}(S,S')$, where $\iota=0$ if\/ $\Gamma^{\sf C}$ has no cycles with $o$ in the interior and 
$\iota=2$ otherwise. 
\end{thm}
\begin{proof}
We use~\eqref{eq:tau-yu}. The relation between the arcs of $\Gamma^{\sf C}$ and the elements of $\cK(S)\cup\cK(S')$ shows that 
$\dim E_S +\dim E_{S'}$ is equal to the number of $\sig$-orbits on the arcs of $\Gamma^{\sf C}$. Let this number be $a$. 
Note that $|\Pi|=r$ is the number of $\sig$-orbits on the vertices of $\Gc$. 
We interpret  $r-a$ in terms of the graph. 

Let $B$ be a connected component of $\Gc$. Suppose first that $B$ is a cycle. If $B$ is not $\sigma$-stable, then 
$\sig$ has $|\cE(B)|=|\cv(B)|$ orbits on $\cE(B)\sqcup \cE(\sigma(B))$ and on $\cv(B)\sqcup\cv(\sigma(B))$.  
If $\sigma(B)=B$, then $\sig$ has 
$$
2+\frac{1}{2}\left(|\cE(B)|-2\right) = 1+\frac{1}{2}|\cE(B)|
$$ 
orbits on $\cE(B)$ by Lemma~\ref{lmC}\,{\sf (iv)}. The number of $\sig$-orbits on $\cv(B)$ is 
$\frac{1}{2}|\cv(B)|=\frac{1}{2}|\cE(B)|$. 
 
Suppose now that $B$ is a segment. If $B$ is not $\sigma$-stable, then  
$\sig$ has $|\cE(B)|=|\cv(B)|-1$ orbits on $\cE(B)\sqcup \cE(\sigma(B))$ and 
$|\cv(B)|$ orbits  
on $\cv(B)\sqcup\cv(\sigma(B))$.  
If $\sigma(B)=B$, then $\sig$ has 
$$
1+\frac{1}{2}\left(|\cE(B)|-1\right) = \frac{1}{2}\left(|\cE(B)|+1\right)
$$ 
orbits on $\cE(B)$ by Lemma~\ref{lmC}\,{\sf (iv)}. The number of $\sig$-orbits on $\cv(B)$ is the same 
$\frac{1}{2}|\cv(B)|=\frac{1}{2}(1+|\cE(B)|)$. 
Summing up, 
\beq \label{part1}
r-a = \frac{1}{2}\#\{\text{the non-$\sigma$-stable segments of $\Gamma^{\sf C}$}\} -  \#\{\text{the $\sigma$-stable cycles of}\ \Gamma^{\sf C}\}. 
\eeq

With obvious changes we repeat the strategy of the proof of Theorem~\ref{thm-A}.
For each connected component $B$ of $\Gamma^{\sf C}$, we have 
$\bar K(B)\subset \langle \esi_i \mid i \in \cv(B)\cup \cv(\sigma(B)), 1\le i\le r\rangle_{\BR}$. 
Hence  a minimal relation among  elements 
of $\cK(S,S')|_{\gt t}$ is (up to signs) a non-trivial linear dependence among 
$\bar\beta(\e)$ with $\e\in \cE(B)$ for some $B$.
If $B$ is a segment, then there are no such dependences. 
 If $B$ is a cycle, then there is one dependence. Since $\bar K(B)=\bar K(\sigma(B))$, the dimension of $(E_S+E_{S'})|_{\gt t}$ is equal to 
\[
\dim E_S + \dim E_{S'} - \#\{\text{the $\sigma$-stable cycles of}\ \Gc\} -\frac{1}{2} \#\{\text{the non-$\sigma$-stable cycles of}\ \Gc\}.
\]
It remains to compare $\dim(E_S+E_{S'})$ and 
$\dim((E_S+E_{S'})|_{\gt t})$. These dimensions are equal if $\delta\not\in E_S+E_{S'}$ and differ by $1$ otherwise. 
Note that $\delta\in E_S+E_{S'}$ if and only if there is a cycle $B$ in $\Gc$ such that 
$\delta\in\langle K(B)\rangle_{\BR}$. 

We say that an arc $(i,\sigma(i))$ with $1\le i\le r$ of $\Gc$ is {\it affine}, if $\beta(i,j)=\delta-2\esi_i$. 
Recall that $\beta(i,j)=\pm(\esi_i-\esi_j)$, if $j\ne \sigma(i)$ and $i\le r$. 
By Lemma~\ref{lmC}, a $\sigma$-stable cycle $B$ contains at most two affine arcs and 
a non-$\sigma$-stable cycle $\tilde B$ contains no affine arcs.  Here $o\not\in F(\tilde B)$ and 
$\delta\not\in \langle K(\tilde B)\rangle_{\BR}$. 

Suppose that $o\in F(B)$. Then one of the $\sigma$-stable arcs of $B$ is affine and another one is not. 
Here $\delta\in \langle K(B)\rangle_{\BR}$. Suppose  that $o\not\in F(B)$.  Then either both $\sigma$-stable arcs of $B$ are not affine and
then clearly $\delta\not\in \langle K(B)\rangle_{\BR}$ or both of them are affine. We consider the latter case. 
Let $i_1,i_2,\ldots,i_m,\sigma(i_m),\ldots,\sigma(i_2),\sigma(i_1)$ be the consecutive vertices of $B$
with $i_j\le r$ for each $j$. Then a relation among $\bar\beta(\e)$ with $\e\in\cE(B)$   looks as follows
\[
2\esi_{i_1} + 2(\esi_{i_2}-\esi_{i_1})+\ldots +2(\esi_{i_m}-\esi_{i_{m-1}})-2\esi_{i_m}=0.
\]
If we replace $2\esi_{i_1}$ with $2\esi_{i_1}-\delta$ and $2\esi_{i_m}$ with $2\esi_{i_m}-\delta$, the result is still zero.
Thus, $\delta\not\in  \langle K(B)\rangle_{\BR}$. 
 
Summing up, $\dim(E_S+E_{S'})=\dim((E_S+E_{S'})|_{\gt t}) + \frac{1}{2}\iota$. We rewrite~\eqref{eq:tau-yu} using~\eqref{part1} and obtain  
\begin{align*}
& \ind\bar\q=1+r+\dim E_S +\dim E_{S'}-2\!\left(\!\dim((E_S+E_{S'})|_{\gt t}) + \frac{\iota}{2}\right)\!=  1{+}r-\!(\dim E_S +\dim E_{S'})\, + \\
 & + 
2\#\{\text{the $\sigma$-stable cycles of}\ \Gc\} + \#\{\text{the non-$\sigma$-stable cycles of}\ \Gc\}-\iota = \\
  &
 = 1 + \frac{1}{2}\#\{\text{the non-$\sigma$-stable segments of $\Gamma^{\sf C}$}\} -  \#\{\text{the $\sigma$-stable cycles of}\ \Gamma^{\sf C}\} + \\
& +  
2\#\{\text{the $\sigma$-stable cycles of}\ \Gc\} + \#\{\text{the non-$\sigma$-stable cycles of}\ \Gc\}-\iota = \\ 
& =   1+  \frac{1}{2}\#\{\text{the non-$\sigma$-stable segments of $\Gamma^{\sf C}$}\}  +  \#\{\text{the cycles of}\ \Gc\} -\iota.
\end{align*}
This finishes the proof. 
\end{proof}

\begin{ex}\label{Cmax}
We consider the case of maximal parabolics, where $S=\widehat{\Pi}\setminus\{\ap_i\}$ and $S'=\widehat{\Pi}\setminus\{\ap_j\}$. 
Let $\Gc=\Gc(S,S')$ be the corresponding graph. 
Each arc of $\Gc$ connects a vertex $u$ with $\sigma(u)$. Thereby $\Gc$ has 
$r$ cycles, where each cycle has two vertices and two $\sigma$-stable arcs. If $i=j$, then 
there are no cycles with $o$ in the interior and $\ind\bar\q=r+1$. In this case, $\bar\q$ is isomorphic to 
$\gt{sp}_{2i}\oplus\gt{sp}_{n-2i}\oplus\bbk$. If $i\ne j$, there are  cycles with $o$ in the interior  and 
$\ind\bar\q=r-1$.  
\end{ex}

\section{A glance at the finite-dimensional case}  \label{fin}

For $\gt g=\gt{sl}_n$, a formula for the index of a seaweed subalgebra $\gt q\subset\g$ in terms of the corresponding 
type-{\sf A} meander graph is obtained in \cite{dk00}.
In \cite{SW-C,mD}, a similar interpretation of $\ind\q$ is given for $\gt{sp}_{2n}$ and $\gt{so}_n$.
We explain how  the TYJ formula, cf~\eqref{eq:tau-yu}, provides a proof for the 
combinatorial statements  in all classical  types,

If $\gt g$ is of type {\sf A} or {\sf C}, then one just simplifies the arguments of Theorems~\ref{thm-A},~\ref{thmC}. 
We have $\ap_0\not\in S$ and $\ap_0\not\in S'$. 
There are no affine arcs, the circle is replaced by a line and its centre disappears.  In particular, in type {\sf A}, each 
cycle in the graph gives rise to a relation among the elements of $\cK(S)\cup-\cK(S')$.

\begin{ex}   \label{ex:old-A} 
Suppose that $n=9$ and $\gt g=\gt{sl}_9$. Choose $S=\{\ap_1,\ap_2,\ap_3,\ap_4,\ap_6,\ap_8\}$ and 
$S'=\{\ap_1,\ap_3,\ap_4,\ap_5,\ap_7,\ap_8\}$. Then
$\bar{\q}=\gt q(S,S')\oplus\bbk d$, where $\gt q(S,S')=:\gt q$ is a seaweed in $\gt g$. 
Clearly $\ind\bar\q=1+\ind\gt q$. 
The type-{\sf A} meander graph $\Gamma_{\sf fin}$ of $\q$ is shown in Figure~\ref{fin-dim}, 
below the dotted line we indicate elements of $\Pi\setminus S$ and above those of $\Pi\setminus S'$. 
\begin{figure}
\setlength{\unitlength}{0.043in}
\raisebox{-12\unitlength}{%
\begin{picture}(96,30)(-5,-11)
\multiput(0,3)(10,0){9}{\circle*{2}}
\put(20,5){\oval(40,20)[t]}
\put(20,5){\oval(20,10)[t]}
\put(55,5){\oval(10,5)[t]}
\put(75,5){\oval(10,5)[t]}
\put(5,1){\oval(10,5)[b]}
\put(35,1){\oval(30,15)[b]} 
\put(35,1){\oval(10,5)[b]}
\put(70,1){\oval(20,10)[b]}
\qbezier[200](0,3),(45,3),(80,3)
{\color{blue}
\put(42.9,-0.5){$\ap_5$}
\put(63.4,-0.5){$\ap_7$}
\put(12.9,4.5){$\ap_2$}
\put(52.9,4.5){$\ap_6$}
}
\end{picture} }
\caption{A graph in a finite-dimensional setting.} \label{fin-dim} 
\end{figure}
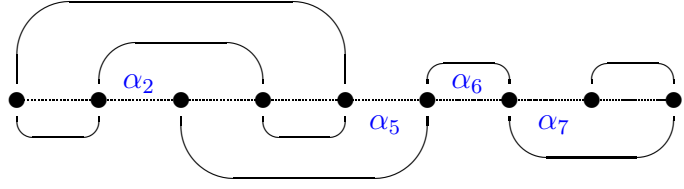
We see that 
\[
\cK(S)=\{\esi_1-\esi_5,\esi_2-\esi_4,\ap_6,\ap_8\} \ \ \text{ and } \ \ 
\cK(S')=\{\ap_1,\esi_3-\esi_6,\ap_4,\esi_7-\esi_9\}.
\] 
There is a natural bijection between $\cK(S)$ and the above arcs, as well as between $\cK(S')$ and the below arcs. 
The unique cycle in the graph gives rise to a relation 
\[
(\esi_1-\esi_5)-\ap_4-(\esi_2-\esi_4)-\ap_1=0
\]
among  elements of $\cK(S)\cup-\cK(S')$.
The TYJ formula asserts that $\ind\gt q$ is equal to 
\begin{align*}
& (n-1)+\#\{\text{the arcs in $\Gamma_{\sf fin}$}\}-2(\#\{\text{the arcs in $\Gamma_{\sf fin}$}\}-
\#\{\text{the cycles in $\Gamma_{\sf fin}$}\})= \\
 & = \#\{\text{the segments of $\Gamma_{\sf fin}$}\} +2\{\text{the cycles in $\Gamma_{\sf fin}$}\} -1.
\end{align*}
This is exactly the statement of \cite{dk00}. In our example, $\ind\gt q=2$.
\end{ex}

\subsection{The orthogonal case} \label{so-fin} Let $\gt q=\gt q(S,S')\subset\gt g$ be a standard seaweed defined by the subsets $S,S'\subset\Pi$. 
A type-{\sf B} meander graph $\Gamma^{\sf B}_{\sf fin}=\Gamma^{\sf B}_{\sf fin}(S,S')$ is constructed by the same principle as in type {\sf C}. 
For $\gt g=\gt{so}_{2r+1}$, 
we put $n=2r$ nodes on a horizontal line.  The arcs of the line are labeled with the simple roots 
$\ap_1,\ap_2,\ldots,\ap_{r-1},\ap_r,\ap_{r-1},\ldots,\ap_1$. 
The graph $\Gamma^{\sf B}_{\sf fin}$ has edges above and below the horizontal line. 

In order to construct the above edges,  we remove the arcs labelled with $\ap_i$ such that $\ap_i\not\in S$ from the line.
The interval $[1,n]$  becomes a union of connected components.  
Let $j_1,\ldots,j_m$ be the consecutive vertices in a connected component $C_k$. For each $1\le i\le m/2$, we draw an edge
$(j_i,j_{m+1-i})$.
The shadow of $(j_i,j_{m+1-i})$ is a segment of $C_k$ connecting $j_i$ with $j_{m+1-i}$
and $\beta(j_i,j_{m+1-i})\in\gt t^*$ is the sum of simple roots belonging to this segment.  
The same procedure applies to  $S'$ and produces the below edges. 
The resulting graph is stable with respect to the reflection $\sigma$, which sends a vertex $i$ to
$n+1-i$.  
Set 
\[\cL(S)=\{\beta(\e) \mid \e \text{ is an above edge of } \Gamma^{\sf B}_{\sf fin}\} \ \ \text{ and } \ \
\cL(S')=\{\beta(\e) \mid \e \text{ is a below edge of } \Gamma^{\sf B}_{\sf fin}\}.
\]
If $\{\ap+{r-1},\ap_r\}\subset S$, then $\cL(S)\ne\cK(S)$, but in any case $\langle\cL(S)\rangle_{\BR} =\langle\cK(S)\rangle_{\BR}$ and 
similarly $\langle\cL(S')\rangle_{\BR} =\langle\cK(S')\rangle_{\BR}$.  
Using the TYJ formula and an argument similar to the proof of Theorem~\ref{thmC}, we obtain 
\[
\ind\gt q= \#\{\text{the cycles of}\ \Gamma^{\sf B}_{\sf fin}\}+\frac{1}{2}\#\{\text{the non-$\sigma$-stable segments of $\Gamma^{\sf B}_{\sf fin}$}\}.
\]
This is the answer of \cite{SW-C}. 

A Dynkin diagram of type {\sf D}$_r$ with $r\ge 4$ has a branching point and this makes a combinatorial interpretation of $\ind\gt q$ much more 
complicated.  Several  type-{\sf D} meander graphs have a crossing (edges meeting not at a vertex). 

The argument of \cite{mD} relies on a reduction to either a parabolic or $\gt q_{\sf ec}:=\gt q(S,S')$, where 
$S=\Pi\setminus\{\ap_r\}$ and $S'=\Pi\setminus\{\ap_{r-1}\}$. The index of a parabolic is computed with the help of the TYJ formula and that computation is quite involved, see \cite[Lemma~4.3]{mD}. The same lemma states that $\ind\q_{\sf ec}=r-2$. The proof of the second statement uses  the  Ra{\"i}s formula for the index of semi-direct products \cite{r}. 
As we show now, it can be replaced by the TYJ formula. 

\begin{ex}\label{ec}
Suppose that $\gt g$ is of type ${\sf D}_r$. 
If $S=\Pi\setminus\{\ap_r\}$ and $S'=\Pi\setminus\{\ap_{r-1}\}$, then 
$|\cK(S)|=|\cK(S')|=\left\lfloor \frac{r}{2}\right\rfloor$. Furthermore,
\[
\cK(S)\cup\cK(S') =\cK(S) \sqcup \{\ap_1+\ap_2+\ldots+\ap_{r-2}+\ap_r\}.
\]
Hence $\ind\q_{\sf ec}=r+2\left\lfloor \frac{r}{2}\right\rfloor-2(\left\lfloor \frac{r}{2}\right\rfloor+1)=r-2$.
\end{ex} 

The TYJ formula can replace also the reduction argument of \cite{mD}. This line of reasoning is quite involved 
and we do not give details here.

\end{document}